\documentclass[final]{siamltex}
\usepackage{newcent}
\usepackage{amsmath}
\usepackage{amssymb}
\usepackage{psfrag,graphicx}
\usepackage{epsfig}
\usepackage{verbatim}
\newtheorem{thm}{Theorem}

\newtheorem{prp}{Proposition}
\newtheorem{ex}{Example}

\usepackage{amsmath,amssymb} 
\usepackage{amssymb}  
\usepackage{tikz}
\usepackage{multicol}
\usepackage{blindtext}
\def\E{\mathbf{E}}

\def\sym{\mathbb{S}}

\def\R{\mathbb{R}}
\def\E{\mathbf{E}}
\def\K{\mathbb{K}}

\title{Multi-Objective Linear Quadratic Team Optimization
\thanks{This work was supported by the Swedish Research Council.}}

\author{Ather Gattami
\thanks{A. Gattami is with the Electrical Engineering School,
        KTH-Royal Institute of Technology, SE-100 44 Stockholm, Sweden.
        {\tt\small gattami@kth.se}
        }}

\begin{document}

\maketitle

\begin{abstract}
In this paper, we consider linear quadratic team problems with an arbitrary number of quadratic constraints in both stochastic and deterministic settings. The team consists of players with different measurements about the state of nature. The objective of the team
is to minimize a quadratic cost subject to additional finite number of quadratic constraints. We will first consider the Gaussian case, where the state of nature is assumed to have a Gaussian distribution, and show that the linear decisions are optimal and can be found by
solving a semidefinite program We then consider the problem of minimizing a quadratic objective for the worst case scenario, subject to an arbitrary number of deterministic quadratic constraints. We show that linear decisions can be found by solving a semidefinite program.
\end{abstract}



\pagestyle{myheadings}
\thispagestyle{plain}
\markboth{ATHER GATTAMI}{Multi-Objective Linear Quadratic Team Optimization}

\section{Introduction}
We consider the problem of distributed decision
making with information constraints under linear quadratic
settings. For instance, information constraints appear naturally
when making decisions over networks. These problems can be
formulated as team problems. The team problem is an optimization problem
with several decision makers possessing different information aiming to optimize a common
objective. Early results in \cite{marschak:1955}  considered static team
theory in stochastic settings and a more general framework was introduced
by Radner \cite{radner}, where existence and uniqueness of solutions where shown.
Connections to dynamic team problems for control purposes where introduced in \cite{ho:chu}. 
In \cite{didinsky:basar:1992}, the team problem
with two team members was solved. The solution cannot be easily
extended to more than two players since it uses the fact that the
two members have common information; a property that doesn't
necessarily hold for more than two players. Also, a nonlinear team
problem with two team members was considered in
\cite{bernhard:99}, where one of the team members is assumed to
have full information whereas the other member has only access to
partial information about the state of the world. Related team
problems with exponential cost criterion were considered in
\cite{krainak:82}. Optimizing team problems with respect to
\textit{affine} decisions in a minimax quadratic cost was shown to
be equivalent to stochastic team problems with exponential cost,
see \cite{fan:1994}. The connection is not clear when the
optimization is carried out over nonlinear decision functions.
The deterministic version (minimizing the worst case scenario) 
of the linear quadratic team decision problem was solved in \cite{gattami:bob:rantzer}.

In this paper, we will consider both Gaussian and deterministic settings(worst case scenario)
for team decision problems under additional quadratic constraints. It's
well-known that additional constraints, although convex, could
give rise to complex optimization problems if the optimized variables are functions
(as opposed to real numbers). For instance linear functions, that is functions of the 
form $\mu(x) = Kx$ where $K$ is a real matrix, are no longer optimal. We will illustrate this fact by the following example:

\begin{ex}
For $x\in \mathbb{R}$, we want to minimize the objective function 
$$|u|^2$$  
subject to $$|x-u|^2\leq \gamma$$

Some Hilbert space theory shows that the optimal $u$ is given by

$$
u=\mu(x) = (|x|-\sqrt{\gamma})x/|x| ~~ if~~  |x|^2>\gamma, 
$$
and
$$
u=\mu(x) = 0 ~~ otherwise. 
$$
Obviously, the optimal $u$ is a nonlinear function of $x$. \\
\end{ex}

Increasing the dimension of $x$, and adding constraints on the structure of $u$, for instance $x\in \mathbb{R}^N$ 
and $u=\mu(x)=(\mu(x_1), .., \mu(x_N))$, certainly makes the constrained optimization more complicated.
The example above shows that, in spite of having a convex optimization carried out 
over a Hilbert space, the optimal decision function is nonlinear. However, we show in the upcoming sections that
multi-objective problems behave nicely when considering the \textit{expected} values of the objectives in the Gaussian case, in the sense that linear decisions are optimal. For the deterministic counterpart which is not an optimization problem over a Hilbert space, we show how to find the linear optimal decisions by semidefinite programming. However, the optimality of the linear decisions remains and open question.

\section{Notation}
The following table gives a list of the notation we are going to usee throughout the text:\\

\begin{tabular}{ll}
$\mathbb{S}^n$& The set of $n\times n$ symmetric matrices.\\
$\mathbb{S}^n_{+}$& The set of $n\times n$ symmetric positive\\
& semidefinite matrices.\\
$\mathbb{S}^n_{++}$& The set of $n\times n$ symmetric positive\\
& definite matrices.\\
$\mathcal{M}$& The set of measurable functions.\\
$\mathcal{C}$& The set of functions
$\mu:\mathbb{R}^p\rightarrow \mathbb{R}^m$ with\\
& $\mu(y)=(\mu_1^T(y_1),\mu_2^T(y_2), ..., \mu_N^T(y_N))^T$,\\
& $\mu_i:\mathbb{R}^{p_i}\rightarrow \mathbb{R}^{m_i}$,
$\sum_{i} m_i=m$, $\sum_{i} p_i=p$.\\
$[A]_{ij}$& The element of $A$ in position $(i, j)$.\\
$\succeq$ & $A\succeq B$ $\Longleftrightarrow$ $A-B\in \mathbb{S}^n_{+}$.\\
$\succ$ & $A\succ B$ $\Longleftrightarrow$ $A-B\in \mathbb{S}^n_{++}$.\\
$\otimes$& The Kronecker binary operation\\
& between two matrices $A$ and $B$, $A\otimes B$.\\
$\mathbf{Tr}$& $\mathbf{Tr}[A]$ is the trace of the matrix $A$.\\
$\mathcal{N}(m,X)$&  The set of Gaussian variables with\\
& mean $m$ and covariance $X$.
\end{tabular}

\section{Linear Quadratic Gaussian Team Theory}
In this section we will review some classical results in
stochastic team theory with new simpler proofs for the linear quadratic case, 
that first appeared in \cite{gattami:mtns06} and \cite{gattami:phd}.

In the static team decision problem, one would like to
solve
\begin{equation}
\label{static}
\begin{aligned}
\min_{\mu}\hspace{2mm} & \mathbf{E}\left[
\begin{matrix}
x\\
u
\end{matrix}
\right]^T
\left[
\begin{matrix}
Q_{xx} & Q_{xu}\\
Q_{ux} & Q_{uu}
\end{matrix}
\right]
\left[
\begin{matrix}
x\\
u
\end{matrix}
\right]\\
\text{subject to } & y_i=C_ix+v_i\\
                   & u_i = \mu_i(y_i)\\
                   & \text{for } i=1,..., N.
\end{aligned}
\end{equation}
Here, $x$ and $v$ are independent Gaussian variables taking values in
$\mathbb{R}^n$ and $\mathbb{R}^{p}$, respectively, with
$x\sim \mathcal{N}(0,V_{xx})$ and $v\sim \mathcal{N}(0,V_{vv})$.
Also, $y_i$ and $u_i$ will be stochastic variables taking values in
$\mathbb{R}^{p_i}$, $\mathbb{R}^{m_i}$, respectively, and
$p_1+...+p_N=p$. We assume that
\begin{equation}
\left[\begin{matrix}
Q_{xx} & Q_{xu}\\
Q_{ux} & Q_{uu}
\end{matrix}\right]\in \mathbb{S}^{m+n},
\end{equation}
and $Q_{uu}\in \mathbb{S}^{m}_{++}$, $m=m_1+\cdots+m_N$.

If full state  information about $x$ is available to each
\textit{decision maker} $u_i$, the minimizing $u$ can be
found easily by completion of squares. It is given by $u=Lx$, where
$L$ is the solution to
$$Q_{uu}L=-Q_{ux}.$$
Then, the cost function in (\ref{static}) can be rewritten as
\begin{equation}
\label{cost}
\begin{aligned}
J(x,u) & = \mathbf{E}\{x^T(Q_{xx}-L^TQ_{uu}L)x\}+\mathbf{E}\{(u-Lx)^TQ_{uu}(u-Lx)\}.
\end{aligned}
\end{equation}
Minimizing the cost function $J(x,u)$, is equivalent to
minimizing $$\mathbf{E}\{(u-Lx)^TQ_{uu}(u-Lx)\},$$
since nothing can be done about
$\mathbf{E}\{x^T(Q_{xx}-L^TQ_{uu}L)x\}$
(the cost when $u$ has full information).

The next theorem is due to Radner \cite{radner},
but we give a different formulation and proof that is simpler, which relies 
on the structure of the linear quadratic Gaussian setting:\vspace{2mm}

\begin{thm}
\label{radner1}
Let $x$ and $v_i$ be Gaussian variables with zero mean,
taking values in $\mathbb{R}^n$ and $\mathbb{R}^{p_i}$, respectively,
with $p_1+...+p_N=p$.
Also, let $u_i$ be a stochastic variable taking values in
$\mathbb{R}^{m_i}$,
$Q_{uu}\in \mathbb{S}^{m}_{++}$, $m=m_1+\cdots+m_N$,
$L\in \mathbb{R}^{m\times n}$, $C_i\in \mathbb{R}^{p_i\times n}$,
for $i=1, ..., N$.
Then, the optimal decision $\mu$ to the optimization problem
\begin{equation}
\label{opt1}
\begin{aligned}
\min_{\mu}\hspace{3mm}   & \mathbf{E}\{(u-Lx)^TQ_{uu}(u-Lx)\}\\
\textup{subject to}\hspace{3mm} & y_i=C_ix+v_i\\
                   & u_i = \mu_i(y_i)\\
                   & \text{for } i=1,..., N.
\end{aligned}
\end{equation}
is unique and linear in $y$.\\
\end{thm}

\begin{proof}
Let $\mathcal{Z}$ be the linear space of functions such that $z
\in \mathcal{Z}$ if $z_i$ is a linear transformation of $y_i$,
that is $z_i=A_i y_i$ for some real matrix $A_i\in
\mathbb{R}^{m_i\times p_i}$.
Since $Q_{uu}\succ 0$, $\mathcal{Z}$ is a linear
space under the inner product
$$\langle g,h \rangle =\mathbf{E}\{g^T Q_{uu}h\},$$
and norm
$$||g||^2=\mathbf{E}\{g^T Q_{uu}g\}.$$
The optimization problem in (\ref{opt1}) where we search for the \textit{linear}
optimal decision can be written as
\begin{equation}
\label{optinner}
\begin{aligned}
\min_{u\in \mathcal{Z}}\hspace{3mm}   & ||u-Lx||^2
\end{aligned}
\end{equation}
Finding the best linear optimal decision $u^*\in \mathcal{Z}$ to the above
problem is equivalent to finding the shortest distance from the
subspace $\mathcal{Z}$ to the element $Lx$, where the minimizing
$u^*$ is the projection of $Lx$ on $\mathcal{Z}$, and hence
unique. Also, since $\mu^*$ is the projection, we have
$$0=\langle u^*-Lx, \mu \rangle=\mathbf{E}\{(u^*-Lx)^TQ_{uu}u\},$$
for all $u\in \mathcal{Z}$. In particular, for
$f_i=(0,0,...,z_i,0,...,0)\in \mathcal{Z}$, we have
$$
\mathbf{E}\{(u^*-Lx)^TQ_{uu}f_i\} =
\mathbf{E}\{[(u^*-Lx)^TQ_{uu}]_i z_i\}=0.
$$
The Gaussian assumption implies that $[(u^*(y)-Lx)^TQ_{uu}]_i$
is independent of $z_i=A_iy_i$, for all linear transformations
$A_i$. This gives in turn that $[(u^*-Lx)^TQ_{uu}]_i$ is
independent of $y_i$. Hence, for any decision $\mu\in
\mathcal{M}\cap \mathcal{C}$, \textit{linear or nonlinear}, we
have that
\begin{equation*}
\begin{aligned}
\mathbf{E}(u^*-Lx)^TQ_{uu}\mu(y)=
\sum_i \mathbf{E}\{[(u^*-Lx)^TQ_{uu}]_i \mu_i(y_i)\}=0,
\end{aligned}
\end{equation*}
and
\begin{equation*}
\begin{aligned}
\mathbf{E}(\mu(y) - & Lx)^T  Q_{uu}(\mu(y)-Lx) \\ 
&= 	\mathbf{E}(u^*-Lx+\mu(y)-u^*)^TQ_{uu}(u^*-Lx+\mu(y)-u^*)\\
&=  	\mathbf{E}(u^*-Lx)^TQ_{uu}(u^*-Lx)+
 	\mathbf{E}(\mu(y)-u^*)^TQ_{uu}(\mu(y)-u^*) \\
&+	2\mathbf{E}(u^*-Lx)^TQ_{uu}(\mu(y)-u^*)\\
&=	\mathbf{E}(u^*-Lx)^TQ_{uu}(u^*-Lx)+
 	\mathbf{E}(\mu(y)-u^*)^TQ_{uu}(\mu(y)-u^*)\\
&\geq \mathbf{E}(u^*-Lx)^TQ_{uu}(u^*-Lx)
\end{aligned}
\end{equation*}
with equality if and only if $\mu(y)=u^*$. This
concludes the proof.\\
\end{proof}

\begin{prp}
\label{radner2} Let $x$ and $v_i$ be independent Gaussian
variables taking values in $\mathbb{R}^n$ and $\mathbb{R}^{p_i}$,
respectively with $x\sim \mathcal{N}(0,V_{xx})$, $v\sim
\mathcal{N}(0,V_{vv})$. Also, let $u_i$ be a stochastic variable
taking values in $\mathbb{R}^{m_i}$, $m=m_1+\cdots + m_N$,
$Q_{xu}\in \mathbb{R}^{n\times m}$, $Q_{uu}\in
\mathbb{S}^{m}_{++}$, $C_i\in \mathbb{R}^{p_i\times n}$, and
$L=-Q_{uu}^{-1}Q_{ux}$. Set $y_i=C_ix+v_i$. Then, the optimal
solution $K_1, ..., K_N$ to the optimization problem
\begin{equation}
\begin{aligned}
\min_{K_i}\hspace{3mm} & \mathbf{E}(u-Lx)^TQ_{uu}(u-Lx)\\
\text{subject to}\hspace{3mm} &  u_i = K_iy_i\\
& \text{for } i=1,..., N.
\end{aligned}
\end{equation}
is the solution of the linear system of equations
\begin{equation}
\label{teamsol}
\begin{aligned}
&\sum_{j=1}^{N}
     [Q_{uu}]_{ij}K_j(C_jV_{xx}C_i^T+[V_{vv}]_{ji})
  = -[Q_{ux}]_{i}V_{xx}C_i^T,\hspace{4mm}\text{for } i= 1, ..., N.
\end{aligned}
\end{equation}
\end{prp}
\vspace{2mm}
\begin{proof}
Let $K=\text{diag}(K_1, ..., K_N)$ and
$
C=\left[\begin{matrix}
C_1^T & \cdots & C_N^T
\end{matrix}\right]^T.
$
The problem of finding the optimal linear feedback law $u_i=K_iy_i$ can be
written as
\begin{equation}
\begin{aligned}
\min_{K_i}\hspace{2mm} &
\mathbf{Tr}[\mathbf{E} \{Q_{uu}(u-Lx)(u-Lx)^T\}]\\
\text{subject to } & u = K(Cx+v)
\end{aligned}
\end{equation}
Now
\begin{equation}
\label{f(k)}
\begin{aligned}
f(K) &= \mathbf{Tr}\left.\mathbf{E}
       \{ Q_{uu}(u-Lx)(u-Lx)^T\}\right.\\
     &= \mathbf{Tr}\left.\mathbf{E}\{
     Q_{uu}(KCx+Kv-Lx)(KCx+Kv-Lx)^T\}\right.\\
     &= \mathbf{Tr}\left[\mathbf{E}\{
       Q_{uu}K(C x x^T C^T+v v^T)K^T-
     2Q_{uu}L x x^T C^T K^T\right.\\
     &\left.+Q_{uu}L x x^T L^T+2Q_{uu}(KC-L)xv^T K^T\}\right]\\
     &= \mathbf{Tr}\left[
     Q_{uu}K(CV_{xx}C^T+V_{vv})K^T-
     2Q_{uu}LV_{xx}C^TK^T+Q_{uu}LV_{xx}L^T)\right]\\
     &= \mathbf{Tr}\left[\sum_{i,j=1}^{N}
     [Q_{uu}]_{ij}K_j(C_jV_{xx}C_i^T+[V_{vv}]_{ji})K_i^T-
     2\sum_{i,j=1}^{N}[Q_{uu}]_{ij}L_jV_{xx}C_i^TK_i^T\right]\\
     &+\mathbf{Tr}[Q_{uu}LV_{xx}L^T].
\end{aligned}
\end{equation}
A minimizing $K$ is obtained by solving $\nabla_{K_i} f(K)=0$:
\begin{equation}
\label{almost}
\begin{aligned}
0 &= \nabla_{K_i} f(K)\\
  &= 2\sum_{j=1}^{N}
     [Q_{uu}]_{ij}K_j(C_jV_{xx}C_i^T+[V_{vv}]_{ji})-
  2\sum_{j=1}^{N}[Q_{uu}]_{ij}L_jV_{xx}C_i^T.
\end{aligned}
\end{equation}

Since $Q_{uu}L=-Q_{ux}$, we get that
\[\begin{array}{c}
\sum_{j=1}^N[Q_{uu}]_{ij}L_jV_{xx}C_i^T=
-[Q_{ux}]_{i}V_{xx}C_i^T,
\end{array}\]
and the equality in (\ref{almost}) is equivalent to
\[\begin{array}{c}
\sum_{j=1}^{N}
     [Q_{uu}]_{ij}K_j(C_jV_{xx}C_i^T+[V_{vv}]_{ji})
  = -[Q_{ux}]_{i}V_{xx}C_i^T,
\end{array}\]
and the proof is complete.
\end{proof}

In general, separation does not hold for the static team problem
when constraints on the information available for every decision
maker $u_i$ are imposed. That is, the optimal decision is
\textit{not} given by $u_i=L\hat{x}_i$, where $\hat{x}_i$ is the
optimal estimated value of $x$ by decision maker $i$. We show it
by considering the following example.

\begin{ex}
Consider the team problem
\begin{equation*}
\begin{aligned}
\text{minimize } & \mathbf{E}\left[
\begin{matrix}
x\\
u
\end{matrix}
\right]^T
\left[
\begin{matrix}
Q_{xx} & Q_{xu}\\
Q_{ux} & Q_{uu}
\end{matrix}
\right]
\left[
\begin{matrix}
x\\
u
\end{matrix}
\right]\\
\text{subject to } & y_i=C_ix+v_i\\
                   & u_i = \mu_i(y_i)\\
                   & \text{for } i=1,..., N
\end{aligned}
\end{equation*}
The data we will consider is:\\
\begin{equation*}
\begin{aligned}
&N=2, \hspace{2mm} C_1=C_2=1, \hspace{2mm} x\sim \mathcal{N}(0,1),\hspace{1mm} v_1\sim \mathcal{N}(0,1),\hspace{1mm}
v_2\sim \mathcal{N}(0,1)\\
&Q_{xx}=1,\hspace{2mm} Q_{uu}=
\left[
  \begin{matrix}
    2 & 1\\
    1 & 2
  \end{matrix}
\right],\hspace{2mm}
Q_{xu}=Q_{ux}^T=
-\left[
  \begin{matrix}
    1 & 1
  \end{matrix}
\right]
\end{aligned}
\end{equation*}
The best decision with full information is given by
\begin{equation*}
\begin{aligned}
u &=-Q_{uu}^{-1}Q_{ux}x=
\left[
  \begin{matrix}
    \frac{2}{3} & -\frac{1}{3}\\
    -\frac{1}{3} & \frac{2}{3}
  \end{matrix}
\right]
\left[
  \begin{matrix}
    1\\
    1
  \end{matrix}
\right] x=
\left[
  \begin{matrix}
    \frac{1}{3}\\
    \frac{1}{3}
  \end{matrix}
\right] x.
\end{aligned}
\end{equation*}
The optimal estimate of $x$ of decision maker 1 is
$$\hat{x}_1=\mathbf{E} \{x|y_1\}=\frac{1}{2}y_1,$$
and of decision maker 2
$$\hat{x}_2=\mathbf{E} \{x|y_2\}=\frac{1}{2}y_2.$$
Hence, the decision where each decision maker combines the best
deterministic decision with her best estimate of $x$ is given by
\begin{equation*}
\begin{aligned}
u_i &=\frac{1}{3}\hat{x}_i=\frac{1}{6}y_i,
\end{aligned}
\end{equation*}
for $i=1,2$. This policy gives a cost equal to $0.611$.
However, solving the team problem yields
$K_1=K_2=\frac{1}{5}$,
and hence the optimal team decision is given by
$$
u_i=\frac{1}{5}y_i.
$$
The cost obtained from the team problem is $0.600$.
Clearly, separation does not hold in team decision problems.
\end{ex}

\section{Team Decision Problems with Power Constraints}
Consider the modified version of the optimization problem (\ref{static}):
\begin{equation}
\label{power}
\begin{aligned}
\min_{\mu}\hspace{2mm} & \mathbf{E}
\left[
\begin{matrix}
x\\
u
\end{matrix}
\right]^T
\left[
\begin{matrix}
Q & S\\
S^T & R
\end{matrix}
\right]
\left[
\begin{matrix}
x\\
u
\end{matrix}
\right]\\
\text{subject to } & y_i=C_ix\\
                   & u_i = \mu_i(y_i)\\
                   & \gamma_i\geq \E \|\mu_i(y_i)\|^2\\
                   & \text{for } i=1,..., N.
\end{aligned}
\end{equation}
The difference from Radner's original formulation is that we have added 
power constraints to the decision functions, $\gamma_i\geq \E \|\mu_i(y_i)\|^2$.

In optimization (minimization) problems, you define the value to be infinite if there doesn't exist any
feasible decision variable that satisfy the constraints. Therefore, usually, one assumes that there is
a feasible point, and hence the value must be finite. Existence conditions are hard to derive usually in spite of 
problems might be convex. So in practice, you run the algorithm and either you get a finite number, or 
it goes indefinitely. Conditions where you a decide whether you have a feasible problem or not are 
of great interest of course. It's a nontrivial problem that is outside the scope of this paper.

In the sequel, we will prove a more general theorem, where we consider power constraints 
on a set of quadratic forms in both the state $x$ and the decision function $\mu$.

\begin{thm}
\label{gattami}
Let $x$ be a Gaussian variable with zero mean and given covariance 
matrix $X$,
taking values in $\mathbb{R}^n$.
Also, let
$
\left[
\begin{matrix}
Q_0 & S_0\\
S_0^T & R_0
\end{matrix}
\right]\in \sym_{+}^{m+n}
$,   $R_0\in \sym_{++}^m$, 
$
\left[
\begin{matrix}
Q_j & S_j\\
S_j^T & R_j
\end{matrix}
\right]\in \sym^{m+n}
$,
and $R_j\in \sym_{+}^m$, 
for $j = 1, ..., M$.
Assume that the optimization problem
\begin{equation}
\label{main}
\begin{aligned}
\min_{\mu\in\mathcal{C}}\hspace{2mm} & \mathbf{E}\left[
\begin{matrix}
x\\
\mu(x)
\end{matrix}
\right]^T
\left[
\begin{matrix}
Q_0 & S_0\\
S_0^T & R_0
\end{matrix}
\right]
\left[
\begin{matrix}
x\\
\mu(x)
\end{matrix}
\right]\\
\textup{subject to } & 
\mathbf{E}\left[
\begin{matrix}
x\\
\mu(x)
\end{matrix}
\right]^T
\left[
\begin{matrix}
Q_j & S_j\\
S_j^T & R_j
\end{matrix}
\right]
\left[
\begin{matrix}
x\\
\mu(x)
\end{matrix}
\right] \leq \gamma_j\\ 
& j=1,..., M
\end{aligned}
\end{equation}
is feasible. Then, linear decisions $\mu$ given by $\mu(x) = K(X) x$, with $K(X)\in \mathbb{K}$, are optimal.\\
\end{thm}
\begin{proof}
Consider the expression
$$
\mathbf{E}
\left[
\begin{matrix}
x\\
\mu(x)
\end{matrix}
\right]^T
\left[
\begin{matrix}
Q_0 & S_0\\
S_0^T & R_0
\end{matrix}
\right]
\left[
\begin{matrix}
x\\
\mu(x)
\end{matrix}
\right]+\sum_{j=1}^M 
\lambda_j
\left( \mathbf{E}
\left[
\begin{matrix}
x\\
\mu(x)
\end{matrix}
\right]
\left[
\begin{matrix}
Q_j & S_j\\
S_j^T & R_j
\end{matrix}
\right]
\left[
\begin{matrix}
x\\
\mu(x)
\end{matrix}
\right]
-\gamma_j
\right).
$$
Take the expectation of a quadratic form with index $j$ to be larger than $\gamma_j$. Then, 
$\lambda_j \rightarrow \infty$ makes the value of the expression above infinite. On the other hand, if  the expectation of a quadratic form with index $j$ is smaller than $\gamma_j$,
then the maximizer $\lambda_j$ is optimal for  $\lambda_j = 0$.

Now let $p^\star$ be the optimal value of the optimization problem (\ref{main}), and consider the objective function
$$
\left[
\begin{matrix}
x\\
u
\end{matrix}
\right]^T
\left[
\begin{matrix}
Q_0 & S_0\\
S_0^T & R_0
\end{matrix}
\right]
\left[
\begin{matrix}
x\\
u
\end{matrix}
\right] = x^T(Q_0-S_0R_0^{-1}S_0^T)x+(u-R_0^{-1}S_0^Tx)^TR_0(u-R_0^{-1}S_0^Tx).
$$
We have that $Q_0-S_0R_0^{-1}S_0^T\succeq 0$, since it's the Schur complement of $R_0$ in the positive 
semi-definite matrix
$
\left[
\begin{matrix}
Q_0 & S_0\\
S_0^T & R_0
\end{matrix}
\right].
$
Since $R_0\succ 0$, a necessary condition for the objective function to be zero is that $u=R_0^{-1}S_0^Tx$, and so $u$ must be linear
(In order for $u$ to have the structure given by $\mathcal{C}$, 
$R_0^{-1}S_0^T$ must be in $\K$, to satisfy the information 
constraints).
 
Now assume that $p^\star>0$. We have
{\small
\begin{equation}
\label{pstar}
\begin{aligned}
p^\star &= \min_{\mu\in\mathcal{C}} \max_{\lambda_i\in \R_+}\hspace{1mm}
\mathbf{E}
\left[
\begin{matrix}
x\\
\mu(x)
\end{matrix}
\right]^T
\left[
\begin{matrix}
Q_0 & S_0\\
S_0^T & R_0
\end{matrix}
\right]
\left[
\begin{matrix}
x\\
\mu(x)
\end{matrix}
\right]+\sum_{j=1}^M 
\lambda_j
\left( \mathbf{E}
\left[
\begin{matrix}
x\\
\mu(x)
\end{matrix}
\right]
\left[
\begin{matrix}
Q_j & S_j\\
S_j^T & R_j
\end{matrix}
\right]
\left[
\begin{matrix}
x\\
\mu(x)
\end{matrix}
\right]
-\gamma_j
\right)\\
&= \min_{\mu\in\mathcal{C}} \max_{\lambda_i\in \R_+}\hspace{1mm}  \mathbf{E}
\left[
\begin{matrix}
x\\
\mu(x)
\end{matrix}
\right]^T
\left(
\left[
\begin{matrix}
Q_0 & S_0\\
S_0^T & R_0
\end{matrix}
\right]+
\sum_{j=1}^M 
\lambda_i
\left[
\begin{matrix}
Q_j & S_j\\
S_j^T & R_j
\end{matrix}
\right]
\right)
\left[
\begin{matrix}
x\\
\mu(x)
\end{matrix}
\right]-\sum_{j=1}^M \lambda_j\gamma_j.
\end{aligned}
\end{equation}
}
Now introduce $\lambda_0$ and the matrix
$$
\left[
\begin{matrix}
Q    	& S\\
S^T & R
\end{matrix}
\right]
=
\sum_{j=0}^M 
\lambda_j
\left[
\begin{matrix}
Q_j & S_j\\
S_j^T & R_j
\end{matrix}
\right],
$$
and consider the minimax problem
\begin{equation}
\label{p0}
\begin{aligned}
p_0 = \min_{\mu\in\mathcal{C}} \hspace{1mm}\max_{\substack{\lambda_j\geq 0\\ \sum_{j=0}^M \lambda_j=1}} \hspace{1mm}  \mathbf{E}
\left[
\begin{matrix}
x\\
\mu(x)
\end{matrix}
\right]^T
\left[
\begin{matrix}
Q & S\\
S^T & R
\end{matrix}
\right]
\left[
\begin{matrix}
x\\
\mu(x)
\end{matrix}
\right]-\sum_{j=1}^M \lambda_j\gamma_j.
\end{aligned}
\end{equation}
Note that a maximizing $\lambda_0$ must be positive, since $\lambda_0=0$ implies that $p_0\leq 0$,
while $\lambda_0>0$ gives  $p_0> 0$. We can always recover the optimal solutions of (\ref{pstar})
from that of (\ref{p0}) by dividing all variables by $\lambda_0$, that is $p^{\star}=p_0/\lambda_0$,
$ \lambda_j \mapsto \lambda_j/\lambda_0$, and $\mu\mapsto \mu/\lambda_0$.
Now we have the obvious inequality ($\min \max \{\cdot\}\geq \max  \min \{\cdot\}$)

\begin{equation*}
\begin{aligned}
p_0 &\geq 
\max_{\substack{\lambda_j\geq 0\\ \sum_{j=0}^M \lambda_j=1}}\min_{\mu\in\mathcal{C}} \hspace{1mm}  \mathbf{E}
\left[
\begin{matrix}
x\\
\mu(x)
\end{matrix}
\right]^T
\left(
\sum_{j=1}^M 
\lambda_j
\left[
\begin{matrix}
Q_j & S_j\\
S_j^T & R_j
\end{matrix}
\right]
\right)
\left[
\begin{matrix}
x\\
\mu(x)
\end{matrix}
\right]-\sum_{j=1}^M \lambda_j\gamma_j.
\end{aligned}
\end{equation*}

For any fixed values of $\lambda_j$, we have $R\succ 0$, so Theorem \ref{radner1} gives the
equality
$$
\min_{\mu\in\mathcal{C}}
\mathbf{E} \left[
\begin{matrix}
x\\
\mu(x)
\end{matrix}
\right]
\left[
\begin{matrix}
Q    	& S\\
S^T & R
\end{matrix}
\right]
\left[
\begin{matrix}
x\\
\mu(x)
\end{matrix}
\right]=
\min_{K\in\mathbb{K}}
\mathbf{E} \left[
\begin{matrix}
x\\
Kx
\end{matrix}
\right]
\left[
\begin{matrix}
Q    	& S\\
S^T & R
\end{matrix}
\right]
\left[
\begin{matrix}
x\\
Kx
\end{matrix}
\right], 
$$
where the minimizing $K$ is unique. Thus,

{\small
\begin{equation*}
\begin{aligned}
p_0 & \geq 
\max_{\substack{\lambda_j\geq 0\\ \sum_{j=0}^M \lambda_j=1}}\min_{\mu\in\mathcal{C}} \hspace{1mm}  \mathbf{E}
\left[
\begin{matrix}
x\\
\mu(x)
\end{matrix}
\right]^T
\left(
\sum_{j=0}^M 
\lambda_j
\left[
\begin{matrix}
Q_j & S_j\\
S_j^T & R_j
\end{matrix}
\right]
\right)
\left[
\begin{matrix}
x\\
\mu(x)
\end{matrix}
\right]-\sum_{j=1}^M \lambda_j\gamma_j\\
&=
\max_{\substack{\lambda_j\geq 0\\ \sum_{j=0}^M \lambda_j=1}}\min_{K\in \K} \hspace{1mm}  \mathbf{E}
\left[
\begin{matrix}
x\\
Kx
\end{matrix}
\right]^T
\left(
\sum_{j=0}^M 
\lambda_j
\left[
\begin{matrix}
Q_j & S_j\\
S_j^T & R_j
\end{matrix}
\right]
\right)
\left[
\begin{matrix}
x\\
Kx
\end{matrix}
\right]-\sum_{j=1}^M \lambda_j\gamma_j.
\end{aligned}
\end{equation*}
}
The objective function is radially unbounded in $K$ since $R \succ 0$. Hence, it can be restricted to
a compact subset of $\K$. Thus, 

\begin{equation*}
\begin{aligned}
p_0 & \geq 
\max_{\substack{\lambda_j\geq 0\\ \sum_{j=0}^M \lambda_j=1}}\hspace{1mm} \min_{K\in \K}  \hspace{1mm}  \mathbf{E}
\left[
\begin{matrix}
x\\
\mu(x)
\end{matrix}
\right]^T
\left(
\sum_{j=0}^M 
\lambda_j
\left[
\begin{matrix}
Q_j & S_j\\
S_j^T & R_j
\end{matrix}
\right]
\right)
\left[
\begin{matrix}
x\\
\mu(x)
\end{matrix}
\right]-\sum_{j=1}^M \lambda_j\gamma_j\\
&=
\min_{K\in \K} \max_{\substack{\lambda_j\geq 0\\ \sum_{j=0}^M \lambda_j=1}} \hspace{1mm}  \mathbf{E}
\left[
\begin{matrix}
x\\
Kx
\end{matrix}
\right]^T
\left(
\sum_{j=0}^M 
\lambda_j
\left[
\begin{matrix}
Q_j & S_j\\
S_j^T & R_j
\end{matrix}
\right]
\right)
\left[
\begin{matrix}
x\\
Kx
\end{matrix}
\right]-\sum_{j=1}^M \lambda_j\gamma_j\\
&\geq \min_{\mu\in\mathcal{C}} \max_{\substack{\lambda_j\geq 0\\ \sum_{j=0}^M \lambda_j=1}}\hspace{1mm}  \mathbf{E}
\left[
\begin{matrix}
x\\
\mu(x)
\end{matrix}
\right]^T
\left(
\sum_{j=0}^M 
\lambda_i
\left[
\begin{matrix}
Q_j & S_j\\
S_j^T & R_j
\end{matrix}
\right]
\right)
\left[
\begin{matrix}
x\\
\mu(x)
\end{matrix}
\right]-\sum_{j=1}^M \lambda_j\gamma_j\\
&= p_0,
\end{aligned}
\end{equation*}
where the equality is obtained by applying Proposition \ref{minmaxtheorem} in the Appendix, 
the second inequality follows from the fact that the set of linear decisions $Kx$, $K\in \K$, is a subset of $\mathcal{C}$, and the second equality follows from the definition of $p_0$. Hence, linear decisions are optimal, and the proof is complete.
\end{proof}

{\bf Remark:} Although Theorem 2 is stated and proved for $y=x$ and
$u=\mu(y)=\mu(x)$, it extends easily to the case $y=Cx$ for any matrix $C$, which often is the case in applications. 
\section{Computation of The Optimal Team Decisions}
The optimization problem that we would like to solve when assuming linear decisions is

\begin{equation}
\label{linopt}
	\begin{aligned}
		\min_{\gamma_0, K \in \K}    &\hspace{2mm} \gamma_0\\
		\text{subject to } &\hspace{2mm}
		\mathbf{E}\left[
		\begin{matrix}
		x\\
		KCx
		\end{matrix}
		\right]^T
		\left[
		\begin{matrix}
		Q_j & S_j\\
		S_j^T & R_j
		\end{matrix}
		\right]
		\left[
		\begin{matrix}
		x\\
		KCx
		\end{matrix}
		\right] \leq \gamma_j, \hspace{2mm} j=0, ..., M,\\
		&\hspace{2mm} x\sim \mathcal{N}(0,X^2).
	\end{aligned}
\end{equation}
Note that we can write the constraints as
\begin{equation}
\begin{aligned}
\mathbf{E}\left[
		\begin{matrix}
		x\\
		KCx
		\end{matrix}
		\right]^T
		\left[
		\begin{matrix}
		Q_j & S_j\\
		S_j^T & R_j
		\end{matrix}
		\right]
		\left[
		\begin{matrix}
		x\\
		KCx
		\end{matrix}
		\right] &=
		\mathbf{E} \left\{\mathbf{Tr} \left[
		\begin{matrix}
		I\\
		KC
		\end{matrix}
		\right]^T
		\left[
		\begin{matrix}
		Q_j & S_j\\
		S_j^T & R_j
		\end{matrix}
		\right]
		\left[
		\begin{matrix}
		I\\
		KC
		\end{matrix}
		\right]		
		xx^T\right\}\\
		&=\mathbf{Tr} X
		\left[
		\begin{matrix}
		I\\
		KC
		\end{matrix}
		\right]^T
		\left[
		\begin{matrix}
		Q_j & S_j\\
		S_j^T & R_j
		\end{matrix}
		\right]
		\left[
		\begin{matrix}
		I\\
		KC
		\end{matrix}
		\right] X,
\end{aligned}
\end{equation}
where we used that $\mathbf{E} xx^T = X^2$.
Hence, we obtain a set of convex quadratic inequalities (convex since $R_j\succeq 0$ for all $j$)
$$
\mathbf{Tr} X
		\left[
		\begin{matrix}
		I\\
		KC
		\end{matrix}
		\right]^T
		\left[
		\begin{matrix}
		Q_j & S_j\\
		S_j^T & R_j
		\end{matrix}
		\right]
		\left[
		\begin{matrix}
		I\\
		KC
		\end{matrix}
		\right] X\leq \gamma_j.
$$
There are many existing computational methods to solve convex quadratic optimization problems (see \cite{boyd:vandenberghe:2004}).

Alternatively, we can formulate the optimization problem as a set of linear matrix inequalities as follows. For simplicity, we will assume that $R_j\succ 0$ for all $j$ (The case $R_j\succeq 0$ is analogue with some technical conditions). 

\begin{thm}
The team optimization problem (\ref{linopt}) is equivalent to 
the semi-definite program
\begin{equation}
\label{sd}
	\begin{aligned}
		\min_{\gamma_0, K \in \K}    	&\hspace{2mm} \gamma_0 \\
		\text{subject to } 	&\hspace{2mm} \mathbf{Tr} P_j \leq \gamma_j \\
						&\hspace{2mm} 0 \preceq 
 	\left[
	\begin{matrix}
		P_j -XQ_jX - XS_jKCX - XC^TK^TS_j^TX & XC^TK^T R_j\\
		R_jKCX							& R_j	
	\end{matrix}
	\right]\\ 
	& \hspace{2mm} j=0, ..., M.
	\end{aligned}
\end{equation}
\end{thm}
\begin{proof}
Introduce the matrices $P_j\in \sym^n$, and write the given constraints as
$$
\gamma_j \geq \mathbf{Tr} P_j
$$
\begin{equation}
	\begin{aligned} 
		P_j-
		X
		\left[
		\begin{matrix}
		I\\
		KC
		\end{matrix}
		\right]^T
		\left[
		\begin{matrix}
		Q_j & S_j\\
		S_j^T & R_j
		\end{matrix}
		\right]
		\left[
		\begin{matrix}
		I\\
		KC
		\end{matrix}
		\right] X \succeq 0.		
	\end{aligned}
\end{equation}
Now we have that
\begin{equation}
	\begin{aligned} 
		0 &\preceq 
		X
		\left[
		\begin{matrix}
		I\\
		KC
		\end{matrix}
		\right]^T
		\left[
		\begin{matrix}
		Q_j & S_j\\
		S_j^T & R_j
		\end{matrix}
		\right]
		\left[
		\begin{matrix}
		I\\
		KC
		\end{matrix}
		\right] X\\ 
		&= P_j -XQ_jX - XS_jKCX - XC^TK^TS_j^TX - XC^TK^T R_j KCX.		
	\end{aligned}
\end{equation}
Since $R_j\succ 0$, the quadratic inequality above can be transformed to a linear matrix inequality using the Schur complement (\cite{boyd:vandenberghe:2004}), which is given by

$$
\left[
	\begin{matrix}
		P_j -XQ_jX - XS_jKCX - XC^TK^TS_j^TX & XC^TK^T R_j\\
		R_jKCX							& R_j	
	\end{matrix}
\right]\succeq 0.
$$
Hence, our optimization problem to be solved is given by
\begin{equation}
	\begin{aligned}
		\min_{K \in \K}    	&\hspace{2mm} \gamma_0 \\
		\text{subject to } 	&\hspace{2mm} \mathbf{Tr} P_j \leq \gamma_j \\
						&\hspace{2mm} 0 \preceq 
 	\left[
	\begin{matrix}
		P_j -XQ_jX - XS_jKCX - XC^TK^TS_j^TX & XC^TK^T R_j\\
		R_jKCX							& R_j	
	\end{matrix}
	\right]\\ 
	& \hspace{2mm} j=0, ..., M,
	\end{aligned}
\end{equation}
which proves our theorem.
\end{proof}


\section{Minimax Team Theory}
\label{minimaxteam}

We considered the problem of static stochastic
team decision in the previous sections. This section
treats an analogous version for the deterministic (or worst case)
problem. Although the problem
formulation is very similar, the ideas of the solution are
considerably different, and in a sense more difficult.

The deterministic problem considered is a quadratic game
between a team of players and nature. Each player has limited
information
that could be different from the other players in the team.
This game is formulated as a minimax problem, where the team is
the minimizer and nature is the maximizer.


\subsection{Deterministic Team Problems}
Consider the following team decision problem
\begin{equation}
\label{minimax}
\begin{aligned}
\inf_{\mu} \sup_{x\neq 0}\hspace{1mm} & \frac{J(x,u)}{||x||^2}\\
\text{subject to }\hspace{1mm} & y_i = C_ix\\
                   & u_i = \mu_i(y_i)\\
                   & \text{for } i=1,..., N
\end{aligned}
\end{equation}
where $u_i\in \mathbb{R}^{m_i}$, $m=m_1+\cdots + m_N$,
$C_i\in \mathbb{R}^{p_i\times n}$.\\
 $J(x,u)$ is a quadratic cost given by
$$J(x,u)=
 \left[
\begin{matrix}
x\\
u
\end{matrix}
\right] ^T
\left[
\begin{matrix}
Q_{xx} & Q_{xu}\\
Q_{ux} & Q_{uu}
\end{matrix}
\right]
\left[
\begin{matrix}
x\\
u
\end{matrix}
\right] ,
$$
where
$$
\small
\left[
\begin{matrix}
Q_{xx} & Q_{xu}\\
Q_{ux} & Q_{uu}
\end{matrix}
\right] \in \mathbb{S}^{m+n}.
$$
We will be interested in the case $Q_{uu}\succ 0$.
The players $u_1$,..., $u_N$ make up a \textit{team}, which plays against
\textit{nature} represented by the vector $x$, using
$\mu\in \mathcal{S}$, that is
$$
\mu(Cx)=\left[
  \begin{matrix}
    \mu_1(C_1x)\\
    \vdots\\
    \mu_N(C_Nx)
  \end{matrix}
\right] .
$$

\begin{thm}
If the value of the game (\ref{minimax}) is equal to
$\gamma^*$, then there is a linear decision $\mu(Cx)=KCx$, with
$K=\text{diag}(K_1,...,K_N)$, achieving that value.
\end{thm}

\begin{proof}
For a proof, consult \cite{gattami:bob:rantzer}.
\end{proof}

\subsection{Relation with The Stochastic Minimax Team Decision Problem}
Now consider the stochastic minimax team decision problem
\begin{equation*}
 \min_{K} \max_{\mathbf{E}\|x\|^2=1}\mathbf{E}
\left\{ x^T\left[
\begin{matrix}
I\\
KC
\end{matrix}
\right]^T
\left[
\begin{matrix}
Q_{xx} & Q_{xu}\\
Q_{ux} & Q_{uu}
\end{matrix}
\right]
\left[
\begin{matrix}
I\\
KC
\end{matrix}
\right]x\right\}.
\end{equation*}
Taking the expectation of the cost in the stochastic
problem above yields the equivalent problem
\begin{equation*}
 \min_{K} \max_{\mathbf{Tr} X=1}\mathbf{Tr}\left[
\begin{matrix}
I\\
KC
\end{matrix}
\right]^T
\left[
\begin{matrix}
Q_{xx} & Q_{xu}\\
Q_{ux} & Q_{uu}
\end{matrix}
\right]
\left[
\begin{matrix}
I\\
KC
\end{matrix}
\right]X
\end{equation*}
where $X$ is a positive semi-definite matrix, and is the
covariance matrix of $x$, i. e. $X=\mathbf{E}\hspace{1mm}xx^T$.
Hence, we see that the stochastic minimax team problem is 
equivalent to the deterministic minimax team problem, where
nature maximizes with respect to all covariance matrices
$X$ of the stochastic variable $x$ with variance 
$\mathbf{E}\hspace{1mm}\|x\|^2=\mathbf{E}\hspace{1mm}x^Tx=
\mathbf{Tr}\hspace{1mm} X = 1$.
\section{Deterministic Team Problems with Quadratic Constraints}
Consider the team problem (\ref{minimax}). An equivalent condition for
the existence of a decision function $\mu^\star\in \mathcal{C}$ that achieves the value of the game $\gamma^\star$ is that 
$$
\left[
\begin{matrix}
x\\
\mu^\star(Cx)
\end{matrix}
\right] ^T 
\left[
\begin{matrix}
Q & S\\
S^T & R
\end{matrix}
\right]
 \left[
\begin{matrix}
x\\
\mu^\star(Cx)
\end{matrix}
\right] \leq \gamma^\star \|x\|^2$$
for all $x$, which is equivalent to
$$
\left[
\begin{matrix}
x\\
\mu^\star(Cx)
\end{matrix}
\right] ^T  
\left[
\begin{matrix}
Q-\gamma^\star I & S\\
S^T & R
\end{matrix}
\right]
\left[
\begin{matrix}
x\\
\mu^\star(Cx)
\end{matrix}
\right] \leq 0$$
for all $x$. This is an example of a \textit{power constraint}. We could also have a set of power
constraints that have to be mutually satisfied. For instance, in addition to the minimization
of the worst case quadratic cost, we could have constraints on the induced norms of the
decision functions

$$\frac{\|\mu_i(C_ix)\|^2}{\|x\|^2}\leq \gamma_i \hspace{4mm} \text{for all } x\neq 0,  \hspace{4mm} 
i=1, ..., M,$$ 
or equivalently given by the quadratic inequalities 

$$\|\mu_i(C_ix)\|^2 - \gamma_i \|x\|^2 \leq 0 \hspace{4mm} \text{for all } x,  \hspace{4mm} 
i=1, ..., M.$$ 
Also, the team members could share a common power source, and the power is proportional to the squared norm of the decisions $\mu_i$:
$$\sum _{i=1}^M \|\mu_i(C_ix)\|^2 - c \|x\|^2 \leq 0 \hspace{4mm} \text{for all } x,  $$ 
for some positive real number $c$.

It's not clear whether linear decisions are optimal, since the example give at the introduction indicates
that, in deterministic settings, nonlinear decision are optimal. However, the next result shows how to obtain the linear optimal decisions by solving a semidefinite program. \\

\begin{thm}
\label{gattami2}
Let $C_i\in \mathbb{R}^{p_i\times n}$,
for $i=1, ..., N$. Let
$
\left[
\begin{matrix}
Q_j & S_j\\
S_j^T & R_j
\end{matrix}
\right]\in \sym^{m+n}
$ for $j=0, ..., M$,  
and $R_j\in \sym_{+}^m$  for $0 = 1, ..., M$.
Then, the set of quadratic matrix inequalities 
\begin{equation}
\begin{aligned}
\left[
\begin{matrix}
x\\
KCx
\end{matrix}
\right]^T
\left[
\begin{matrix}
Q_j & S_j\\
S_j^T & R_j
\end{matrix}
\right]
\left[
\begin{matrix}
x\\
KCx
\end{matrix}
\right] \leq 0  ~~~~\forall x, \hspace{3mm} j=0,..., M,
\end{aligned}
\end{equation}
is equivalent to 
\begin{equation}
\begin{aligned}
\left[
\begin{matrix}
Q_j + S_j KC + C^TK^TS_j^T & C^TK^T R_j\\
R_jKC & -R_j
\end{matrix}
\right]
\preceq 0, \hspace{3mm} i=0, ..., M.
\end{aligned}
\end{equation}

\end{thm}
\begin{proof}
We have the following chain of inequalities:
$$
\left[
\begin{matrix}
x\\
KCx
\end{matrix}
\right]^T
\left[
\begin{matrix}
Q_j & S_j\\
S_j^T & R_j
\end{matrix}
\right]
\left[
\begin{matrix}
x\\
KCx
\end{matrix}
\right] \leq 0
$$
$$
\Updownarrow
$$
$$
\left[
\begin{matrix}
I\\
KC
\end{matrix}
\right]^T
\left[
\begin{matrix}
Q_j & S_j\\
S_j^T & R_j
\end{matrix}
\right]
\left[
\begin{matrix}
I\\
KC
\end{matrix}
\right] \preceq 0
$$

$$
\Updownarrow
$$

$$
Q_j + S_j KC + C^TK^TS_j^T+ C^TK^TR_jKC \preceq 0
$$

$$
\Updownarrow
$$

$$
A=\left[
\begin{matrix}
Q_j + S_j KC + C^TK^TS_j^T & C^TK^T R_j\\
R_jKC & -R_j
\end{matrix}
\right]
\preceq 0,
$$ 

\noindent where the last equivalence follows from taking the Schur complement of 
$R_j$ in $A$ (see \cite{boyd:vandenberghe:2004}). Hence, our optimization problem becomes
\begin{equation}
\begin{aligned}
\left[
\begin{matrix}
Q_j + S_j KC + C^TK^TS_j^T & C^TK^T R_j\\
R_jKC & -R_j
\end{matrix}
\right]
\preceq 0, \hspace{3mm} i=0, ..., M.
\end{aligned}
\end{equation}
This completes the proof.
\end{proof}

\section{Conclusions}
We have studied multi-objective linear quadratic optimization of team decisions
in both stochastic and deterministic settings. Constrained decision problems tend
to have nonlinear optimal solutions. We have shown that for the Gaussian setting,
linear decisions are in fact optimal, and we can find the linear optimal solutions by
solving a semidefinite program. We then explore the problem of finding the linear
optimal decisions for its deterministic counterpart and show that we can find
the optimal solution by solving a semidefinite program. Future work will consider
optimality of the linear decisions in the deterministic framework. Another problem of interest
is an an $\mathcal{S}$-procedure sort of a result, where we want to find decision function $\mu$ 
such that the inequality $J_0(\mu(x),x)\leq 0$ is satisfied if  $J_1(\mu(x),x)\leq 0$, where $J_0, J_1$ are some
quadratic forms in $\mu$ and $x$. However, this is a much harder problem since the search for linear 
function $\mu(x)$ is not a covnex problem, and it's not clear if it can be convexified.
\section{Acknowledgements}

The author is grateful for Prof. Anders Rantzer and Prof. Bo Bernhardsson for discussions on the topic.\\
This work is supported by the Swedish Research Council.
\bibliography{../..//ref/mybib}

\section*{Appendix}

\subsection*{Game theory}
Let $J=J(u,w)$ be a functional defined on a product vector space 
$\mathbb{U}\times \mathbb{W}$, to be minimized by 
$u\in U\subset \mathbb{U}$ and maximized by $w\in W\subset \mathbb{W}$, where
$U$ and $W$ are the constrained sets. This defines a zero-sum game,
with kernel $J$, in connection with which we can introduce two values,
the \textit{upper value}
$$
\bar{J}:=\inf_{u\in U}\sup_{w\in W} J(u,w),
$$  
and the \textit{lower value}
$$
\underline{J}:=\sup_{w\in W}\inf_{u\in U} J(u,w).
$$  
Obviously, we have the inequality $\bar{J}\geq \underline{J}$.
If  $\bar{J}= \underline{J}=J^\star$, then $J^\star$ is called the 
\textit{value} of the zero-sum game. Furthermore, if there exists a
pair $(u^\star\in U, w^\star\in W)$ such that 
$$
J(u^\star,w^\star)=J^\star,
$$
then the pair $(u^\star,w^\star)$ is called a (pure-strategy) 
\textit{saddle-point solution}. In this case, we say that the game
admits a \textit{saddle-point} (in pure strategies). Such a saddle-point
solution will equivalently satisfy the so-called 
\textit{pair of saddle-point inequalities}:
$$
J(u^\star,w)\leq J(u^\star,w^\star)\leq J(u,w^\star),\hspace{3mm} \forall u\in \mathbb{U},
 \forall w\in \mathbb{W}.
$$

\begin{prp}
\label{minmaxtheorem}
Consider a two-person zero-sum game on convex finite dimensional
action sets $U_1\times U_2$, defined by the continuous kernel 
$J(u_1,u_2)$. Suppose that $J(u_1,u_2)$ is strictly convex in 
$u_1$ and strictly concave in $u_2$. Suppose that either
\begin{itemize}
\item[($i$)] $U_1$ and $U_2$ are closed and bounded, or
\item[($ii$)] $U_i\subseteq \mathbb{R}^{m_i}$, $i=1,2$, and 
$J(u_1,u_2)\rightarrow \infty$ as $\|u_1\|\rightarrow \infty$,
and $J(u_1,u_2)\rightarrow -\infty$ as $\|u_2\|\rightarrow \infty$.
\end{itemize}  
Then, the game admits a unique pure-strategy saddle-point 
equilibrium.
\end{prp}
\begin{proof}
 See \cite{basar:olsder:1999}, pp. 177.
\end{proof}

\textbf{Remark}. The assumption of strict convexity and concavity in Proposition \ref{minmaxtheorem}
can be relaxed to only convexity and concavity, and a saddle-point exists in pure strategies, but it is not necessarily unique.

\end{document}